\documentclass[11pt]{article}
\usepackage{amsmath}
\usepackage{amsthm}
\usepackage{amssymb}
\usepackage{graphics}
\usepackage{bm}
\usepackage{setspace}
\usepackage[numbers]{natbib}
\usepackage{authblk}

\usepackage{tikz}
\usepackage{amsfonts}
\usepackage[cal=bickham,calscaled=1.05]{mathalfa}
\usepackage{eucal}
\usepackage{latexsym}
\usepackage{mathdots}
\usepackage{mathrsfs}
\usepackage{mathtools}
\usepackage{hyphenat}
\usepackage{amscd}
\usepackage{enumerate}
\usepackage[pagebackref,hypertexnames=false, colorlinks, citecolor=blue, linkcolor=red, pdfstartview=FitB]{hyperref}
\usepackage[latin1]{inputenc}
\usepackage{tikz}
\usepackage{tikz-cd}
\usetikzlibrary{positioning}
\usepackage[numbers]{natbib}

\usepackage{fullpage}
\usepackage{color}

\newcommand{\be}{\begin{equation}}
\newcommand{\ee}{\end{equation}}

\newcommand{\ba}{\begin{eqnarray*}}
\newcommand{\ea}{\end{eqnarray*}}

\newcommand{\bal}{\begin{align}}
\newcommand{\eal}{\end{align}}
\newcommand{\baln}{\begin{align*}}
\newcommand{\ealn}{\end{align*}}

\newcommand{\bi}{\begin{itemize}}
\newcommand{\ei}{\end{itemize}}

\newcommand{\bn}{\begin{enumerate}}
\newcommand{\en}{\end{enumerate}}

\newcommand{\bbm}{\begin{bmatrix}}
\newcommand{\ebm}{\end{bmatrix}}
\newcommand{\bpm}{\begin{pmatrix}}
\newcommand{\epm}{\end{pmatrix}}
\newcommand{\bsm}{\left ( \begin{smallmatrix}}
\newcommand{\esm}{\end{smallmatrix} \right) }

\newcommand{\bp}{\begin{proof}}
\newcommand{\ep}{\end{proof}}

\newcommand{\scr}{\ensuremath{\mathscr}}

\newcommand{\mc}{\ensuremath{\mathcal}}
\newcommand{\mf}{\ensuremath{\mathfrak}}
\newcommand{\mb}{\ensuremath{\mathbb }}

\newcommand{\Ga}{\ensuremath{\Gamma}}

\newcommand{\om}{\ensuremath{\omega}}

\newcommand{\Addresses}{{
  \bigskip

Fouad~Naderi, \textsc{Department of Mathematics, University of Manitoba, Winnipeg, Manitoba, R3T 2M6 Canada}\par\nopagebreak
\textit{E-mail address:} \texttt{naderif@myumanitoba.ca}

Yong~Zhang, \textsc{Department of Mathematics, University of Manitoba, Winnipeg, 
Manitoba, R3T 2M6 Canada}\par\nopagebreak
\textit{E-mail address:} \texttt{Yong.Zhang@umanitoba.ca}

\vspace{1cm}

}}

\newtheorem{thm}{Theorem}

\newtheorem{lemma}{Lemma}

\newtheorem{cor}{Corollary}
\newtheorem*{thm*}{Theorem}

\theoremstyle{definition}
\newtheorem{defn}{Definition}
\newtheorem{remark}{Remark}
\newtheorem{eg}{Example}

\title{Generalized amenability in quantum groups}

\usepackage{authblk}

\author[1]{Fouad Naderi\thanks{Email: \texttt{naderif@myumanitoba.ca}}}
\author[2]{Yong Zhang\thanks{Email: \texttt{Yong.Zhang@umanitoba.ca}}}

\affil[1,2]{University of Manitoba}

\date{\vspace{-1.2cm}}

\begin{document}
\maketitle
\begin{abstract}
We study generalized amenability of locally compact quantum groups $\mathbb{G}=(\scr M, \Delta, \phi, \psi)$. 
Let $C_r^*(\mathbb{G})$ and $VN(\mathbb{G})$ denote, respectively, the C*-algebra and von Neumann algebra generated by the quantum left regular representation of $\mathbb{G}$. 
With natural covariance assumptions and a traceability condition on $VN(\mathbb{G})$, we show that if $C_r^*(\mathbb{G})$ is approximately amenable, then $\mathbb{G}$ admits an invariant quantum mean. 
We show further that the same conclusion holds if the predual algebra $\scr M_*$ has a bounded approximate identity and is approximately amenable.
\end{abstract}

\section{Introduction}
Amenability theory can be traced back to the work of Lebesgue and Banach on measure theory. 
In \cite{Neumann29}, von Neumann observed that the well-known Banach-Tarski paradox -- which demonstrates that there is no finitely additive measure defined on all subsets of $\mathbb{R}^n$ that is invariant under Euclidean motions (translations, rotations, and reflections) when $n \geq 3$ -- stems from the fact that the free group $\mathbb{F}_2$ on two generators lacks a property that he called \emph{amenability}.
Following von Neumann, Mahlon Day \cite{Day_amn.sg} established the general theory of amenability on groups. 

Let $G$ be a locally compact group, and $L^{\infty}(G)$ the Banach space of all essentially bounded Haar-measurable functions on $G$.  A positive linear functional $\mu$ on $L^{\infty}(G)$ is a {\it mean} if $\|\mu\| =\mu(1)=1$. 
A mean $\mu$ is  a left invariant mean (LIM) if $\mu(L_x f)=\mu(f)$ for all $x\in G$, where $L_x$ ($x\in G$) is the left translation operator on $L^{\infty}(G)$ defined by $L_x f(y) = f(xy)$ for all $y \in G$.
The group $G$ is called \emph{amenable} if there exists a LIM on $L^{\infty}(G)$. 

Johnson \cite{John_cohom} introduced the notion of amenability for Banach algebras and showed that a locally compact group $G$ is amenable if and only if the associated convolution algebra $L^1(G)$ is amenable as a Banach algebra. 

Let $A$ be a Banach algebra, and $X$ be a Banach $A$-bimodule. 
A linear map $D: A \to X$ is a {\it derivation} if it satisfies the Leibniz product rule
    $$D(ab)=D(a) \cdot b + a\cdot D(b), \qquad \forall\, a, b \in A.$$
For any $x \in X$,  the mapping 
  $ad_{x} : A \to X$ defined by
        $ad_{x}(a) = a \cdot x - x \cdot a$  is a continuous derivation, called an {\it inner derivation}. 
        
 Given a Banach $A$-bimodule $X$,  its dual space $X^*$ is natunally  a Banach $A$-bimodule with module actions defined by
\[
\langle  f\cdot a, ~ x \rangle = \langle f, ~a \cdot x \rangle, \quad 
\langle  a\cdot f, ~ x \rangle = \langle f, ~x \cdot a \rangle
\]
for $a\in A, f\in X^*$ and $x\in X$.
Here $\langle \cdot, \cdot \rangle$ denotes the bilinear mapping of the duality action $X^* \times X \to \mathbb C$ speciafied by $\langle  f, ~ x \rangle = f(x)$. 

Johnson's amenability is defined as follows: 
\begin{defn} \label{amenable_B.alg}
    A Banach algebra $A$ is {\it amenable} if for each Banach $A$-bimodule $X$, every continuous derivation $D: A \to X^*$ is inner.
\end{defn}

Avoiding involving $X^*$,  Gourdean \cite{Gourdeau-amn} showed the following.
\begin{thm}\label{Th: Gourdean}
The Banach algebra $A$ is amenable if and only if for each Banach $A$-bimodule $X$, every continuous derivation $D$: $A\to X$ can be approximated pointwise by a net $(ad_{x_i})$ of inner derivations with $(x_i)\subset X$ being uniformly bounded.
\end{thm}

Johnson's seminal work sparked extensive research on amenability in Banach and operator algebras, leading to a rich and profound theory with numerous developments and applications (see \cite{ Dales_book} and \cite{Con-amn-nucl, ES-AKM, DQV-ABPC, Hag-nucl-amn}). 

Dropping the boundedness requirement of the net $(x_i)$ in Theorem~\ref{Th: Gourdean}, the notion of approximate amenability was introduced and investigated in  \cite{GL_Gen.amn, GLZ_Gen.amn} (see also \cite{CGZ_psd_amn, DLZ06, GS_app_psd_AG, GZ_psd}).

For a Banach algebra $A$, the projective tensor product $A{\otimes}_\gamma A$ is a Banach $A$-bimodule under the canonical actions specified by:
\[ 
    a\cdot (b\otimes c)=a\,b \otimes c, ~~\text{and} ~~ (b\otimes c)\cdot a = b \otimes c\,a, \quad \forall\, a, b, c \in A.
\]
We recall that the norm of $A{\otimes}_\gamma A$ is given by
$$\|u\|_\gamma = \mathrm{inf}\left\{ \sum_{i=1}^{\infty} \|a_i \| \| b_i\| ~ :~  u = \sum_{i=1}^{\infty} a_i \otimes b_i \right\}, \qquad (u \in A{\otimes}_\gamma A).$$
The {\it multiplication operator} $\pi$: $A{\otimes}_\gamma A\to A$ is the linear mapping determined by:
\[
    \pi(a\otimes b)=a\,b \quad (a,b\in A).
\]
The first and the second dual operators of $\pi$ are denoted by $\pi^*$: $A^* \to (A{\otimes}_\gamma A)^*$ and $\pi^{**}$: $(A{\otimes}_\gamma A)^{**} \to A^{**}$. They are all $A$-module homomorphisms.

\begin{defn} \label{appr-amn-defn}
Let $A$ be a Banach algebra.
\begin{enumerate} 
    \item $A$ is \emph{approximately amenable} if for each Banach $A$-bimodule $X$ and any continuous derivation $D: A \to X$,  there is a net $(x_i)_{i \in I}\subset X$ (possibly unbounded) such that
$$
D(a) = \lim_{i} ~(a \cdot x_i - x_i \cdot a) \quad  (a \in A)
$$
in the norm topology of $X$.
\item  $A$ is \emph{weak* approximately amenable} if for each Banach $A$-bimodule $X$ and any continuous derivation $D: A \to X^*$ there is a net $(f_i)_{i \in I}$ in $ X^*$ such that
$$
D(a) =\text{w*-} \lim_{i} ~ (a \cdot f_i - f_i \cdot a) \quad ( a \in A),
$$
where w* denotes the weak* topology of $X^*$.
\item $A$ is \emph{pseudo-amenable} if there exists a net $(u_i)\subset A \,{\otimes}_\gamma\, A$ (possibly unbounded),  called an \emph{approximate diagonal} for $A$, such that
$$\lim_i a \cdot u_i - u_i \cdot a = 0
    \quad\text{and}\quad
   \lim_i \pi(u_i) a = a \quad (a\in A)$$
in the norm topology of the respective spaces.
\end{enumerate}
\end{defn}

We refer the reader to \cite{GL_Gen.amn,GLZ_Gen.amn,GZ_psd,Zh_sol.unsol} for examples of these notions and their interrelations. 
We highlight here that, by \cite{GLZ_Gen.amn}, approximate amenability and weak* approximate amenability are equivalent. 
Moreover,  if $A$ has a bounded approximate identity,  then approximate amenability of $A$ implies pseudo-amenability of $A$. 
The converse also holds,  as was recently clarified in \cite{G-S-Z}.
 Inparticular, a unital Banach algebra is approximately amenable if and only if it is pseudo-amenable.


A question posed by Ghahramani, Zhang, and others asks whether an approximate amenable C*-algebra must be amenable (see \cite[Question 9.4]{GL_Gen.amn}, \cite[Section 5]{GZ_psd}, \cite[Questions 16 \& 17]{Zh_sol.unsol}, and \cite[Question 3.5]{Zhang-app-psd-amn}).

We study C*-algebras associated with quantum groups, focusing on the relationship between the approximate amenability of such algebras and the amenability of the underlying quantum groups.

In the remaider of this section, we briefly review some necessary notions from von Neumann algebras and quantum groups. The standard references for us are \cite{Davidson_C*, Kadison_Fundamentals_II, Murphy,Takesaki_I, Takesaki_II} and \cite{ES-Kac-alg, VD-qg-vN, KV-qg-lcptq,KV-qg-vN, QG-Timmer, Tomatsu-amn-dscr-qg}.


     Let $\mathscr{M}$ be a von Neumann algebra and $\mathscr{M}^{+}$ be its cone of positive elements. A function $\phi : \mathscr{M}^{+} \longrightarrow [0, +\infty]$ is  a {\it weight} if it satisfies the following two conditions.
     \begin{enumerate}
         \item 
         $\phi(x+y)=\phi(x)+\phi(y)$ for $x, ~ y \in \mathscr{M}^{+}$;
         \item 
         $\phi(\lambda x)=\lambda \phi(x)$ for $\lambda \geq 0$ and $x \in \mathscr{M}^{+}$.
     \end{enumerate}
 Given a weight $\phi$, the following sets are important to us:
    \begin{itemize}
        \item $\mathfrak{p}_{\phi}=\{ x \in \mathscr{M}^{+} \, : \, \phi(x) < \infty \} $, and the linear span of $\mathfrak{p}_{\phi}$ is denoted by $\mathfrak{m}_{\phi}$.
        \item  $\mathfrak{n}_{\phi} =\{  x \in \mathscr{M} \, : \, \phi(x^* x) < \infty \}$, which is a left ideal in $\mathscr{M}$. 
     \end{itemize}
     
In fact, $\mathfrak{m}_{\phi} =\text{span}(\mathfrak{n}_{\phi}^{*}\mathfrak{n}_{\phi}) =\{ \sum_{j=1}^{n} y_{j}^{*}x_j \, : \, n\in \mathbb{N}, ~ x_j, y_j \in \mathfrak{n}_{\phi} \}$, and it is a $\ast$-subalgebra of $\mathscr{M}$ such that $\mathfrak{m}_{\phi} \cap \mathscr{M}^{+}=\mathfrak{p}_{\phi}$ (see  \cite[Lemma VII-2.1]{Takesaki_II} and \cite[P. 842]{KV-qg-lcptq}).  
One  can extend the weight $\phi$ to a linear functional on $\mathfrak{m}_{\phi}$. 
We call the weight $\phi$
     \begin{itemize}
     \item {\it finite} if $\mathfrak{p}_{\phi}=\mathscr{M}^{+}$, which is equivalent to $\phi(1) < \infty$;
     \item {\it semifinite} if $\mathfrak{m}_{\phi}$ is weak* dense in $\mathscr{M}$, which is equivalent to saying that every nonzero $x \in \mathscr{M}^{+}$ majorizes some $y \in \mathscr{M}^{+}$ such that  $y\neq 0$ and $\phi(y) < \infty$;
     \item {\it faithful} if $\phi(x) \neq 0$ for all $x \in \mathscr{M}^{+}$ with $x\neq 0$;
     \item {\it normal} if it has the monotone convergence property for  bounded increasing nets, i.e., for any bounded increasing net $(x_i)\subset \mathscr{M}^{+}$, we have $\phi(\mathrm{sup}_i ~ x_i)=\mathrm{sup}_i ~\phi(x_i)$;
     \item  {\it tracial } if it  satisfies 
          $\phi(x^*x)=\phi(xx^*)$ for all $ x \in \mathscr{M}$.
     \end{itemize}

  A Hopf-von Neumann algebra is  a von Neumann algebra $\mathscr{M}$ equipped with a co-multiplication $\Delta: \mathscr{M} \longrightarrow \mathscr{M} \overline{\otimes} \mathscr{M} $,  which is a normal, unital and injective $\ast$-homomorphism satisfying 
\[
(\Delta \otimes \mathrm{id})\Delta = (\mathrm{id} \otimes \Delta)\Delta.
\]  
 Here, $\overline\otimes$ denotes the von Neumann algebra tensor product and $\mathrm{id}$ denotes the identity map on $\mathscr{M}$. We will write $(\mathscr{M}, \Delta)$ for the Hopf-von Neumann algebra.  
 Let $\mathscr{M}_{*}$ be the predual of $\mathscr{M}$. Then $\Delta$ determines a preadjont $\Delta_{*}$: $\mathscr{M}_{*} \otimes_{\gamma} \mathscr{M}_{*} \to\mathscr{M}_{*}$ defined by $\langle \Delta_*(\Phi), m\rangle = \langle \Phi, \Delta(m)\rangle$ for $\Phi\in \mathscr{M}_{*} \otimes_{\gamma} \mathscr{M}_{*}$ and $m\in \mathscr M$.
This preadjoint yields a product $\star$ on $\mathscr M_*$ defined by
\[
   \langle f \star g ,~ m \rangle =\langle \Delta_{*}( f \otimes g) ,~ m\rangle , \quad (f,g\in \mathscr M_*,  m \in \mathscr{M}).
\]
With this product, $\mathscr{M}_{*}$ becomes a  Banach algebra.
Let $1 = 1_\mathscr{M}$ denote the identity of $\mathscr{M}$.  Then
\begin{equation} \label{multiplicative-unit}
    \langle f \star g ,~ 1 \rangle = \langle f, 1 \rangle \langle g, 1 \rangle,
\end{equation}
 i.e.  $1$ is multiplicative on $(\mathscr{M}_{*}~, \star)$.

 Let $\mu \in \mathscr{M}^{*}$ be a state on $(\mathscr{M}, \Delta)$.
We call $\mu$ a {\it left invariant quantum mean (LIQM)} 
    (or simply a {\it left invariant mean (LIM)}) if
   $$\mu\big((\omega \otimes \mathrm{id})\Delta(F)\big)
    = \omega(1)\,\mu(F)\,
    \qquad (F \in \mathscr{M},~ \omega \in \mathscr{M}_{*}),$$
    which is equivalent to 
 $\omega \cdot \mu = \omega(1)\mu$ for all $\omega \in \mathscr{M}_{*}$, where $\cdot$ denotes the module action of $\mathscr{M}_*$ on $\mathscr M^*$.

We call $\mu$ a {\it right invariant quantum mean (RIQM)} 
    (or simply a {\it right invariant mean (RIM)}) if
    $$\mu\big((\mathrm{id} \otimes \omega)\Delta(F)\big)
    = \mu(F)\,\omega(1),
    \qquad (F \in \mathscr{M},~ \omega \in \mathscr{M}_{*}),$$
    which is equivalent to
   $\mu \cdot \omega = \omega(1)\mu$ for all $\omega \in \mathscr{M}_{*}$.

Denote by $\chi$ the flip map  on $\mathscr{M} \overline{\otimes}  \mathscr{M}$, which  is  defined by $\chi(a\otimes b)=b\otimes a$. for $a, b\in \mathscr M$.  Given a Hopf-von Neumann algebra $(\mathscr{M}, \Delta)$, a co-involution on $\mathscr{M}$ is a unital involutive  normal anti-automorphism $\kappa$ of order 2  that satisfies
$$(\kappa \otimes \kappa)\circ \Delta = \chi \circ \Delta \circ \kappa. $$ 
If such co-involution exists, we call the triple $(\mathscr{M}, \Delta, \kappa)$ 
a {\it co-invollutive Hopf-von Neumann algebra}.  In this case, $\mathscr{M}_{*}$ becomes an involutive Banach algebra with the involution $\omega^*$ defined by
$$
\langle \omega^{*}, m \rangle :=\overline{\langle \omega, \kappa(m)^* \rangle } \quad  (\omega \in \mathscr{M}_{*} ~, ~  m \in \mathscr{M}), 
$$
where $\kappa(m)^*$ is the usual adjoint of the element $\kappa(m) \in \mathscr{M}$.  We refer the reader to \cite[Section 1.2]{ES-Kac-alg} for details,  and we note that in general, $\kappa$ is distinct from the natural adjoint of the von Neumann algebra.  For example, Let $G$ be a locally compact group. then the function algebra $L^{\infty}(G)$ is a Hopf-von Neumann algebra with the co-multiplication defined by 
\begin{equation}\label{E: Delta for G}
\Delta(F)(s,t) = F(st) \quad (F\in L^\infty(G) \quad s,t\in G). 
\end{equation}
Define 
\begin{equation}\label{E:kappa for G}
\kappa(F)(t)=F(t^{-1})  \quad (F \in L^{\infty}(G),  t \in G).
\end{equation}

 Then, $\kappa$ is a co-involution on $L^{\infty}(G)$ \cite[Example 1.2.9]{ES-Kac-alg}. We have $\omega^*(t)=\Delta_G(t)^{-1}\overline{\omega(t^{-1})}$ for $\omega \in L^1(G)$ and $t \in G$, where $\Delta_G$ is the modular function of $G$.

  A {\it Kac algebra} is a quadruple $\mathbb{K}=(\mathscr{M}, \Delta, \kappa, \phi)$ (\cite[Definition 2.2.1]{ES-Kac-alg} and \cite[P. 40]{QG-Timmer}),  where $(\mathscr{M}, \Delta, \kappa)$ is a co-involutive Hopf-von Neumann algebra, and $\phi$ is a normal, faithful, semi-finite weight on $\mathscr{M}^+$ with the following {\it left invariance property}.
 $$(\mathrm{id} \otimes \phi)\Delta(x)=\phi(x) 1 \qquad (x \in \mathfrak{p}_{\phi}).$$
 The Kac algebra $\mathbb{K}$ is {\it discrete} if $\mathscr{M}_{*}$ is unital, and it is  {\it  unimodular} if $\phi$ is invariant under $\kappa$, i.e., if $\phi \circ \kappa(x) =\phi(x)$ for all $x\in \scr M^+$. 
 From \cite[Corollary 6.3.4]{ES-Kac-alg},  a discrete Kac algebra $\mathbb{K}=(\mathscr{M},  \Delta, \kappa, \phi)$ is unimodular, and the predual $\scr M_*$ is a Banach $\ast$-algebra under its involution induced by $\kappa$ \cite[P. 15-16]{ES-Kac-alg}.

 The following lemma  is readily checked.

 \begin{lemma} \label{unit-of-predual}
  Let $\mathbb{K}=(\mathscr{M},  \Delta, \kappa, \phi)$  be a discrete Kac algebra.  If $\omega \in \mathscr{M}_{*}$ is a positive linear functional on $\mathscr{M}$ and is unitary in the sense that $\omega \star \omega^{*}=\omega^{*} \star \omega = 1$, then $\omega$ is a normal state on $\mathscr{M}$, i.e., we have $\|\omega\|=\omega(1)=1$.
\end{lemma}
 
In the above and throughout the paper, we occasionally abuse the notation $1$ to denote different units, such as the identity of the predual, the identity of the von Neumann algebra, or the scalar unit. The intended meaning should be clear from the context.

For a discrete Kac algebra $\mathbb{K} = (\mathscr{M}, \Delta, \kappa, \phi)$, we will use the following notations:
\begin{itemize}
    \item $\mathcal{U}(\mathscr{M}_*)^+$: the set of unitary elements of $\mathscr{M}_*$ which are positive linear functionals on $\mathscr{M}$;
    \item $\mathcal{S}(\mathscr{M}_*)^+$: the set of normal states on $\mathscr{M}$.  
\end{itemize}
By Lemma \ref{unit-of-predual}, we have that
$\mathcal{U}(\mathscr{M}_*)^+ \subseteq \mathcal{S}(\mathscr{M}_*)^+$.

From \cite{KV-qg, KV-qg-vN},
a {\it locally compact quantum group} is a quadruple $\mathbb{G}=(\mathscr{M}, \Delta, \phi, \psi)$ for which $(\mathscr{M}, \Delta)$  is a Hopf-von Neumann algebra,  and $\phi$ and  $\psi$ are normal faithful semi-finite weights on $\mathscr{M}^{+}$ that are, respectively, left and right invariant in the sense that
\[
\phi[(\omega\otimes\mathrm{id})\Delta(x)]= \omega(1) \phi(x), \quad  \psi[(\mathrm{id} \otimes \omega)\Delta(x)]=\omega(1) \psi(x)
 \] 
 for all $\omega \in \mathscr{M}_{*}^+$ and all $x\in \scr M^+$. 
 We call  such $\phi$ (resp. $\psi$) a {\it quantum left (resp. right) Haar weight}.
$\mathbb G$ is called compact  if its left Haar weight $\phi$ is bounded, i.e., $\phi$ is a LIM defined on the whole $\mathscr M$.


 For a locally compact group $G$,  $L^{\infty}(G)$ and the group von Neumann algebra $VN(G)$ are typical examples of locally compact quantum groups. 
 In general, any Kac algebra $\mathbb{K} = (\mathscr{M}, \Delta, \kappa, \phi)$ is  a locally compact quantum group by the setting $\mb G =(\mathscr{M}, \Delta, \phi , \phi\circ\kappa)$ (see \cite[Chapter 6]{ES-Kac-alg} and \cite{KV-qg}).

Given a locally compact quantum group $\mathbb{G}=(\mathscr{M}, \Delta, \phi, \psi)$,  
consider the canonical GNS construction with respect to the normal faithful semifinite weight $\phi$. The equation
 $$\langle x , | \, y\rangle = \phi(y^* x) \quad (x, y \in \mathfrak{n}_{\phi})$$ 
defines an inner product on $\mf{n}_\phi$. Completing it, we obtain a Hilbert space $\scr H_\phi$. 
The multiplication operator $\pi_\phi$: $\scr M \to B(\scr H_\phi)$ defined by 
$$\pi_\phi(a)(b) = ab \quad (a\in \scr M, b\in \mf{n}_\phi)$$ 
is a faithful, normal $\ast$-representation of $\scr M$, called the GNS representation  (\cite[Proposition VII-1.4]{Takesaki_II}). We have $\mathscr{M}\cong \pi_\phi(\mathscr{M})$, and we may regard $(\scr M,\scr H_\phi, \pi_\phi)$ as the standard form of $\scr M$. 
In $\scr H_\phi$, an element $x$ of $\mf n_\phi$ will be denoted by $\Lambda_\phi(x)$.
By \cite[Theorem 1.2]{KV-qg-vN}, there is a unitary operator
$ W$ on $\scr H_\phi \otimes_2 \scr H_\phi$, 
called the {\it multiplicative unitary} 
for $\mathbb{G}$,
such that $W\in \mathscr{M}\otimes \mathscr{B}(\scr H_\phi)$ and
$$ W^*[\Lambda_\phi(x) \otimes \Lambda_\phi(y)] = (\Lambda_\phi \otimes \Lambda_\phi)[\Delta(y)(x\otimes 1)], \quad \forall x,y \in \mathfrak{n}_{\phi}\,.$$

As a concrete example, consider $\mathscr{M}=L^\infty(G)$, where $G$ is a locally compact group.  Let $\phi$ be the weight on $\scr M$ defined by the Haar measure on $G$. Then $\scr H_\phi = L^2(G)$ and the multiplicative unitary for $\scr M$ is $W_G$: $L^2(G \times G)=L^{2}(G)\otimes_2 L^{2} G)\to  L^2(G \times G)$ with
\[
      W_G[h](s, \,t)= h(s,\, s^{-1}t)
\]
for all $h\in L^2(G\times G)$ and all $s,\,t \in G$.

\section{Covariant quantum groups}\label{sec:quantum-self-covariant-dynamical-systems}

Let $\mathbb{G}=(\mathscr{M}, \Delta, \phi, \psi)$ be a locally compact quantum group and $W$ the multiplicative unitary of $\mb G$.
By \cite[Proposition~3.17]{KV-qg} (see also \cite[P. 71]{KV-qg-vN}),  we have $\Delta(x)=W^*(1\otimes x)W$ for all $x \in \scr M$. 
Moreover, by \cite[P. 74]{KV-qg-vN}, one can define an injective algebra homomorphism
$ \lambda$: $\scr M_* \to \mathscr{B}(\scr H_\phi)$ such that
\[
   \lambda(\omega) = (\omega \otimes \mathrm{id})(W),
\]
where $\omega \otimes \mathrm{id}: \mathscr{M}\otimes \mathscr{B}(\scr H_\phi) \to \mathscr{B}( \scr H_\phi)$ is the slicing map defined by  $ \omega \otimes \mathrm{id}(a\otimes T)=\omega(a) \,T $ for all $a\in \mathscr{M}$ and $T\in \mathscr{B}(\scr H_\phi)$.  
Here we note that, in general, $\lambda$ may not be a $*$-homomorphism.
We call $\lambda$ the {\it left regular representation} of $\scr M_*$ associated to $\mb G$, and we
denote by $C_{r}^{*}(\mb G)$ and $VN(\mb G)$ the $C^{*}$-algebra and the von Neumann algebra generated by the range of $\lambda$, respectively. 

A family 
$\Ga\subset \mathcal{S}(\mathscr{M}_*)^+$ 
is called {\it regular} if $\operatorname{span}(\Gamma)$ is norm dense in $\scr M_*$ and $\lambda(\Ga)$ contains only unitary elements of $C_{r}^{*}(\mathbb{G})$. 
This condition immediately implies that $C_{r}^{*}(\mathbb{G})$ is unital, and hence $\mb G$ must be discrete.

For a discrete Kac algebra $(\mb K, \Delta, \kappa, \phi)$,  it is clear that $\Ga =  U(\scr M_*)^+$ is regular.

\begin{defn}\label{defn: quantum-regular-groups}
Let $\Ga\subset \mc S(\scr M_*)^+$ be regular.
We say that the quantum group $\mathbb{G}$ is {\it left-covariant over $\Ga$} if the following two conditions are satisfied.
\begin{enumerate}[(i)]
\item    For each $\omega\in\Gamma$, the left action $L_\om$ on $\mathscr{M}$ defined by  $L_\omega(F) =\omega\cdot F$
is a normal $\ast$-automorphism.
\item   The {\it left intertwining} identity
\begin{equation}\label{left-Gamma-covariance}
\pi_{\phi}(\omega\cdot F)=
\lambda(\omega)\pi_{\phi}(F)\lambda(\omega)^{*} \qquad (F\in\mathscr{M},\ \omega\in\Gamma)
\end{equation}
holds.
\end{enumerate}

Similarly, we say that $\mathbb{G}$ is {\it right-covariant over $\Ga$} if the following conditions are satisfied.
\begin{enumerate}[(i)']
\item    For  $\omega\in\Gamma$, the right action $R_\om$ on $\mathscr{M}$ defined by  $R_\omega(F) =F\cdot \omega$
is a normal $\ast$-automorphism.
\item  The right intertwining identity
\begin{equation}\label{right-Gamma-covariance}
    \pi_{\phi}(F\cdot\omega)
    =
    \lambda(\omega)^{*}\pi_{\phi}(F)\lambda(\omega)
    \qquad (F\in\mathscr{M},\ \omega\in\Gamma)
\end{equation}
holds.
\end{enumerate}
\end{defn}


\begin{eg} \label{KIDS-example}
For a discrete group $G$, let $\mathscr{M}=\ell^{\infty}(G)$. Then $\mathscr{M_{*}}=\ell^1(G)$. 
With $\Delta$ and $\kappa$ being defied by \eqref{E: Delta for G} and  \eqref{E:kappa for G} and the translation-invariant normal faithful semi-finite weight  implemented by the counting measure $\phi$, it is well-known that $(\ell^\infty(G), \Delta, \kappa, \phi)$ is a discrete Kac algebra and hence is a discrete quantum group.
The co-multiplication $\Delta$ induces the convolution product
$$(\omega*\eta)(s)=\sum_{t\in G}\omega(t)\eta(t^{-1}s)$$
 on $\ell^1(G)$. 
 The induced algebra $\scr M_*$ is indeed a Banach $*$-algebra with the involution given by
  $$\omega^*(s)=\overline{\omega(s^{-1})}.$$
The GNS representation of $\scr M$ is $\pi_{\phi}: \ell^{\infty}(G) \longrightarrow \mathscr{B}(\ell^{2}(G))$ given by the multiplication operation,
while the left regular representation of $\scr M_*$ is the unital $*$-homomorphism $\lambda$: $\ell^1(G) \to \mathscr{B}( \ell^{2}(G))$ formulated by
 \be\label{E: lambda}
     (\lambda(\omega)\xi)(s)=(\omega *\xi)(s) = \sum_{t\in G}\omega(t)\,\xi(t^{-1}s)
  \quad(\omega\in\ell^1(G),\ \xi\in\ell^2(G),\ s\in G).
\ee
Let $\Ga = \{\delta_g : g \in G\}$. Then $\Ga \subset \mc S(\ell^1(G))^+$, 
and $\Ga$ is  regular since $\Ga\subset \mc U(\ell^1(G))^+$  and $\lambda$ is a $*$-homomorphism.
In fact,  $\Ga = \mc U(\ell^1(G))^+$ (see \cite{Naderi-thesis-gen-amn} for a proof).
We show that this quantum group 
is right-covariant over $\Ga$.
  
It is readily checked that for $F\in\ell^\infty(G)$ and $g\in G$, 
$$
(F\cdot \delta_g)(s)=F(gs)
\qquad (s\in G).$$
This ensures that 
the right actions of $\delta_g$ on $\ell^\infty(G)$ is  normal $*$-automorphisms.

 We now verify the right intertwining identity.
For $F \in \ell^\infty(G)$, $\xi \in \ell^2(G)$, and $g,s \in G$, we have
$$(\pi_\phi(F)\xi)(s) = F(s)\xi(s), \qquad 
(\lambda(\delta_g)\xi)(s) = \xi(g^{-1}s), \qquad 
(\lambda(\delta_g)^*\xi)(s) = \xi(gs).$$
So
\begin{align*}
[\lambda(\delta_g)^*\pi_\phi(F)\lambda(\delta_g)\xi](s)
&=
[\pi_\phi(F)\lambda(\delta_g)\xi](gs) \\
&=
F(gs)\,(\lambda(\delta_g)\xi)(gs) \\
&=
F(gs)\,\xi(g^{-1}gs) \\
&=
F(gs)\,\xi(s)\\
&=[F \cdot\delta_g ](s) \xi(s)\\
&=[\pi_\phi(F \cdot\delta_g)\xi](s).
\end{align*}
Therefore,  the right intertwining identity  holds. 
Thus, $(\ell^\infty(G), \Delta, \kappa, \phi)$ is a right- covarant quantum group over $\Gamma$.
\end{eg}

\begin{remark}
On $\ell^\infty(G)$,  we may consider the co-multiplication $\Delta': \ell^{\infty}(G) \to \ell^{\infty}(G\times G)$ defined by
 $$\Delta'(f)(s,t)=f(ts) \qquad (s,t\in G).$$
  Still, we choose the counting measure $\phi$ on $G$ for the left and right invariant weights, and $\Ga = \{\delta _g: g\in G\}$.
    Then a similar argument will conclude that the induced quantum group $(\ell^\infty(G), \Delta',\kappa,\phi)$ is left-covariant over $\Ga$. In this case, the multiplicative unitary is $W'$: $
 \ell^2(G \times G)   \to  
 \ell^2(G \times G)$ with
\[
      W'[h](s, t)= h(s,\, ts^{-1}) \quad (h\in \ell^2(G \times G), s,\,t \in G),
\]
and the left regular representation $\lambda'$: $\ell^1(G) \to \scr B(\ell^2(G))$ is given by
\[
     (\lambda'(\omega)\xi)(s)=(\xi *\omega)(s) = \sum_{t\in G}\xi(t)\,\omega(t^{-1}s)
  \quad(\omega\in\ell^1(G),\ \xi\in\ell^2(G),\ s\in G).
\]
In fact, if $\hat{G}$ denotes the group $G$ equipped with the reverse product $s\odot t=ts$, then
\[
(\ell^\infty(G),\Delta',\kappa,\phi)\cong(\ell^\infty(\hat G),\Delta,\kappa,\phi).
\]
\end{remark}

We now recall some facts from \cite[Sections III.3 and III.4]{Takesaki_I} concerning the von Neumann algebra $\scr M$ and its predue $\scr M_*$.
\begin{enumerate}
\item  Since the multiplication in $\mathscr{M}$ is weak* separately continuous,  $\mathscr{M}_*$ is an $\mathscr{M}$-bimodule. 
\item Given $\omega \in \mathscr{M}_*^+$, there is a unique projection $p \in \mathscr{M}$, called the \emph{support projection} of $\omega$, such that 
$$
p\cdot \omega = \omega \cdot p = p\cdot \omega \cdot p = \omega,
$$
and $\omega$ is faithful on the reduced von Neumann algebra $p\mathscr{M}p$.   We denote this $p$ by $s(\omega)$.

\item For each $\Omega\in\mathscr{M}_*$, there are a unique partial isometry $v\in\mathscr{M}$ and a unique $\omega\in\mathscr{M}_*^+$ such that the polar decomposition
$$\Omega=v \cdot  \omega, \qquad v^*v=s(\omega)$$
holds. We have  $\|\Omega\|=\|\omega\|$, and we denote $\omega$ by $|\Omega|$. 
\end{enumerate}

\begin{lemma}\label{essntial-inequality-quantum-covariant}
Let $\mathbb{G}=(\mathscr{M}, \Delta, \phi, \psi)$ be a locally compact quantum group and  $\Gamma\subseteq \mathcal{S}(\mathscr{M}_{*})^+$ be regular.
\begin{enumerate}[(i)]

\item If $\mathbb{G}$ is left-covariant over $\Ga$, then for $\omega\in\Gamma$ and  $\eta\in\mathscr{M}_{*}$, we have
\begin{equation}\label{modulus-left-Gamma}
|\eta\star\omega|
=|\eta|\star\omega,
\end{equation}
and
\begin{equation}\label{ineq-modulus-left-Gamma}
\| \, |\eta|\star\omega-|\eta|\, \|
\leq
2\|\eta\|^{1/2}
\|\eta\star\omega-\eta\|^{1/2}.
\end{equation}

\item If $\mathbb{G}$ is right-covariant over $\Ga$, then for $\omega\in\Gamma$ and $\eta\in\mathscr{M}_{*}$,we have
\begin{equation}\label{modulus-right-Gamma}
|\omega\star\eta|
=\omega\star|\eta|,
\end{equation}
and
\begin{equation}\label{ineq-modulus-right-Gamma}
\|\,\omega\star|\eta|-|\eta|\,\|
\leq
2\|\eta\|^{1/2}
\|\omega\star\eta-\eta\|^{1/2}.
\end{equation}
\end{enumerate}
\end{lemma}

\begin{proof}
We prove the left covariance case. The right covariance case can be proved similarly.

Let $\eta = v|\eta|$ be the polar decomposition of $\eta\in\mathscr{M}_{*}$.
For
$\omega\in\Gamma$, 
since $L_\omega$ is $*$-automorphism, we have $L_\omega(\scr M^+) = \scr M^+$, and there is a partial isometry $u\in \scr M$ such that $v=L_\omega(u)$.
From the definition, $\eta\star\omega = \eta\circ L_\omega$, and $|\eta|\star\omega = |\eta|\circ L_\omega\in \scr M_*^+$.  To obtain \eqref{modulus-left-Gamma}, it suffices to show that
\[
|\eta\circ L_\omega| = |\eta|\circ L_\omega.
\]
Denote $\rho = |\eta|\circ L_\omega$.
For  $F\in\mathscr{M}$, we have
\begin{align*}
(\eta\circ L_{\omega})(F)
&=
\eta(L_{\omega}(F))                                      
=
|\eta|(L_{\omega}(F)v)                                    
=
|\eta|\big(L_{\omega}(F)L_{\omega}(u)\big)  \\
&=
(|\eta|\circ L_{\omega})(Fu)                      
=
\rho(Fu)                                     
=
(u\cdot\rho)(F).
\end{align*}
Thus,
$$
\eta\circ L_{\omega}
=u\cdot\rho. 
$$
By uniqueness of the polar decomposition,
$$
|\eta\circ L_{\omega}| = \rho
=|\eta|\circ L_{\omega}.
$$
Therefore, \eqref{modulus-left-Gamma} holds.

To show \eqref{ineq-modulus-left-Gamma}, we first recall the  Kosaki's inequality \cite{Kosaki84}
\begin{equation}\label{Kosaki-inequality-Gamma}
\|\,|\theta|-|\eta|\,\|^{2}
\leq
2\,\|\theta+\eta\|\,\|\theta-\eta\| 
\end{equation}
for $\theta,\eta\in\mathscr{M}_{*}$.
Set $\theta=\eta\star\omega$ 
We have $\|\theta\| \leq \|\eta\|$. Then 
$$
\|\theta+\eta\| \leq \|\theta\|+\|\eta\|\leq 2\|\eta\|.$$
combining this with \eqref{modulus-left-Gamma} and \eqref{Kosaki-inequality-Gamma}, we obtain \eqref{ineq-modulus-left-Gamma}.
\end{proof}


\section{Approximate amenability of covariant quantum groups}\label{sec:approximate-amenability-discrete-quantum-groups}

In this section, we study approximate amenability in covariant quantum groups. We show that a covariant quantum group has an invariant quantum mean if $C_r^*(\mb G)$ or $\scr M_*$ is approximately amenable. This extends some known results of Ghahramani et al concerning groups.


\begin{thm}\label{Gamma-covariant-implies-RIM}
Let $\mathbb{G}=(\mathscr{M}, \Delta, \phi, \psi)$ be a locally compact quantum group and  $\Gamma\subseteq \mathcal{S}(\mathscr{M}_{*})^+$ be regular. Suppose that $VN(\mathbb{G})$ admits a normal tracial state $\tau$ and that $C_{r}^{*}(\mathbb{G})$ is approximately amenable.

\begin{enumerate}[(i)]
\item If $\mathbb{G}$ is left covariant over $\Gamma$, then $\mathscr{M}$ admits a right invariant mean. \label{left covariant}

\item If the $\mb G$ is right covariant over $\Gamma$,  then $\mathscr{M}$ admits a left invariant mean. \label{right covariant}

\end{enumerate}
\end{thm}

\begin{proof}
We first note that $C_{r}^{*}(\mathbb{G})$ is unital, and its identity is the identity of $\mathscr{B}(\mathscr{H}_{\phi})$. 
Since $C_{r}^{*}(\mathbb{G})$ is approximately amenable, it is pseudo-amenable. 
Hence there exists an approximate diagonal
$$
d_{\alpha}
=
\sum_{n=1}^{r_{\alpha}} a_{n}^{\alpha}\otimes b_{n}^{\alpha}
\in
C_{r}^{*}(\mathbb{G})\otimes_{\gamma} C_{r}^{*}(\mathbb{G}),
\qquad r_{\alpha}<\infty,
$$
such that
\begin{equation}\label{normalized-approx-diagonal}
\pi(d_{\alpha})
=\sum_{n=1}^{r_{\alpha}}a_{n}^{\alpha}b_{n}^{\alpha}=1.
\end{equation}
Here, $\pi$: $C_r^*(\mb G)\otimes_\gamma C_r^*(\mb G) \to C_r^*(\mb G)$ is the multiplication map. 
Let $\mathscr{H}=\mathscr{H}_{\phi}$. 
Extend $\pi$ to the multiplication map
$$
\Pi:
\mathscr{B}(\mathscr{H})\otimes_{\gamma}\mathscr{B}(\mathscr{H})
\longrightarrow
\mathscr{B}(\mathscr{H}).
$$
Since the inclusion $C_{r}^{*}(\mathbb{G})\hookrightarrow \mathscr{B}(\mathscr{H})$ is contractive, we may regard $d_{\alpha}$ as an element of $\mathscr{B}(\mathscr{H})\otimes_{\gamma}\mathscr{B}(\mathscr{H})$ with $\Pi(d_{\alpha})=1$. Since $\tau$ is a normal state on $VN(\mathbb{G})$, it has a normal state extension $\widetilde{\tau}$ to $\mathscr{B}(\mathscr{H})$. So
$$
\widetilde{\tau}|_{VN(\mathbb{G})}=\tau,
\qquad
\|\widetilde{\tau}\|=\|\tau\|=\tau(1)=1.
$$
Let
$\chi$: $\mathscr{B}(\mathscr{H})\otimes_{\gamma}\mathscr{B}(\mathscr{H})
\to
\mathscr{B}(\mathscr{H})\otimes_{\gamma}\mathscr{B}(\mathscr{H})$
be the flip map.
For each $\alpha$, define a normal functional $\Omega_{\alpha}$: $\mathscr{M} \to \mathbb{C}$ by
\[
    \langle \Omega_{\alpha}(F) = \langle \tilde{\tau}, ~  \Pi \circ \chi \left(\pi_\phi(F)\cdot d_{\alpha}\right) \, \rangle
    = \langle  \tilde{\tau},~  \sum_{n=1}^{r_\alpha} b_{n}^{\alpha}\pi_\phi (F) a_{n}^{\alpha} \rangle .
\]
Noting that $\|\Pi\|=\|\chi\|=\|\pi_{\phi}\|=\|\widetilde{\tau}\|=1$, we have
$$\|\Omega_{\alpha}(F)\|
\leq \|F\|\,\|d_{\alpha}\| \qquad (F\in\mathscr{M}). $$
On the other hand, using \eqref{normalized-approx-diagonal} and the fact that $\tau$ is tracial on $VN(\mathbb{G})$, we get
\[
\|\Omega_{\alpha}\|
\geq \, |\Omega_{\alpha}(1)|
=
\left\langle
\widetilde{\tau},\, \Pi\circ\chi(d_{\alpha})
\right\rangle
=
\left\langle \tau,\, \sum_{n=1}^{r_{\alpha}} b_{n}^{\alpha}a_{n}^{\alpha}
\right\rangle 
=
\left\langle \tau,\,\sum_{n=1}^{r_{\alpha}} a_{n}^{\alpha}b_{n}^{\alpha}
\right\rangle 
=
\langle\tau,\,1\rangle=1.
\]
Thus
\begin{equation}\label{Omega-alpha-norm-bounds}
1\leq \|\Omega_{\alpha}\|\leq \|d_{\alpha}\|.
\end{equation}
For any $T \in \mathscr{B}(\mathscr{H})$ and $c \in C_{r}^{*}(\mathbb{G})$, we have
\begin{align}\label{switching-identity}
   \Pi \circ \chi (T \cdot d_{\alpha} \cdot c)&=\Pi \circ \chi (\sum Ta_{n}^{\alpha}\otimes b_{n}^{\alpha}c)
    =\Pi (\sum b_{n}^{\alpha}c \otimes Ta_{n}^{\alpha})
    = \sum b_{n}^{\alpha}[c Ta_{n}^{\alpha}] \notag \\
    &=\Pi \circ \chi (\sum [c T]a_{n}^{\alpha} \otimes b_{n}^{\alpha}))
    = \Pi \circ \chi (cT \cdot d_{\alpha}).
\end{align}

Assume that $\mathbb{G}$ is left covariant over $\Gamma$. Fix $\omega\in\Gamma$ and $F\in\mathscr{M}$. Taking $c=\lambda(\omega)$ and $T=\pi_{\phi}(F)\lambda(\omega)^{*}$ in \eqref{switching-identity}, we obtain
\begin{equation}\label{left-displacement}
\Pi\circ\chi
\bigl(
\lambda(\omega)\pi_{\phi}(F)\lambda(\omega)^{*}\cdot d_{\alpha}
\bigr)
=
\Pi\circ\chi
\bigl(
\pi_{\phi}(F)\lambda(\omega)^{*}\cdot d_{\alpha}\cdot \lambda(\omega)
\bigr).
\end{equation}
Using the intertwining identity \eqref{left-Gamma-covariance}, we compute
\begin{align*}
(\Omega_{\alpha}\star\omega-\Omega_{\alpha})(F)
&=
\Omega_{\alpha}(\omega\cdot F)-\Omega_{\alpha}(F) \\
&=
\left\langle
\widetilde{\tau},
\Pi\circ\chi
\bigl(
\pi_{\phi}(\omega\cdot F-F)\cdot d_{\alpha}
\bigr)
\right\rangle \\
&=
\left\langle
\widetilde{\tau},
\Pi\circ\chi
\bigl(
(\lambda(\omega)\pi_{\phi}(F)\lambda(\omega)^{*}
-
\pi_{\phi}(F))\cdot d_{\alpha}
\bigr)
\right\rangle .
\end{align*}
By \eqref{left-displacement} and the identity $\lambda(\omega)^{*}\lambda(\omega)=1$, we get
\begin{align*}
(\Omega_{\alpha}\star\omega-\Omega_{\alpha})(F)
&=
\left\langle
\widetilde{\tau},
\Pi\circ\chi
\bigl(
\pi_{\phi}(F)\lambda(\omega)^{*}\cdot d_{\alpha}\cdot\lambda(\omega)
-
\pi_{\phi}(F)\cdot d_{\alpha}
\bigr)
\right\rangle \\
&=
\left\langle
\widetilde{\tau},
\Pi\circ\chi
\bigl(
\pi_{\phi}(F)\lambda(\omega)^{*}\cdot
[d_{\alpha}\cdot\lambda(\omega)
-
\lambda(\omega)\cdot d_{\alpha}]
\bigr)
\right\rangle .
\end{align*}
Therefore,
$$\|\Omega_{\alpha}\star\omega-\Omega_{\alpha}\|
\leq \|d_{\alpha}\cdot\lambda(\omega)
- \lambda(\omega)\cdot d_{\alpha}\|.$$
Since $d_{\alpha}$ is an approximate diagonal for $C_{r}^{*}(\mathbb{G})$ and $\lambda(\omega)\in C_{r}^{*}(\mathbb{G})$, we have
\begin{equation}\label{Omega-left-limit}
\|\Omega_{\alpha}\star\omega-\Omega_{\alpha}\|
\longrightarrow 0,
\qquad \omega\in\Gamma.
\end{equation}
Replacing $\Omega_{\alpha}$ by $\Omega_{\alpha}/\|\Omega_{\alpha}\|$, and using \eqref{Omega-alpha-norm-bounds}, we may assume in \eqref{Omega-left-limit} that
 $\|\Omega_{\alpha}\|=1$. Let
$$\Omega_{\alpha}=v_{\alpha}\cdot\mu_{\alpha}$$
be the polar decomposition of $\Omega_{\alpha}$. Then $\mu_{\alpha}\in(\mathscr{M}^*)^+$ and
$$
\|\mu_{\alpha}\|=\mu_{\alpha}(1)=\|\Omega_{\alpha}\|=1.
$$

By Lemma~\ref{essntial-inequality-quantum-covariant}, for every $\omega\in\Gamma$,
\[
\|\mu_{\alpha}\star\omega-\mu_{\alpha}\|
\leq 2\|\Omega_{\alpha}\star\omega-\Omega_{\alpha}\|^{1/2}.
\]
We therefore have
\begin{equation}\label{mu-left-limit}
\|\mu_{\alpha}\star\omega-\mu_{\alpha}\|
\longrightarrow 0,
\qquad \omega\in\Gamma.
\end{equation}
The net $(\mu_{\alpha})$ lies in $\mc S(\scr M^*)^+$ which is
 weak$^{*}$-compact. 
 Passing to a subnet if necessary, we may assume that
$$
\mu_{\alpha}\xrightarrow{w^{*}}\mu
$$
where $\mu\in(\mathscr{M}^*)^+$ is a state.  By \eqref{mu-left-limit} we have
\[
\mu\star\omega=\mu = \omega(1) \mu
\qquad (\omega\in\Gamma),
\]
since every $\omega\in \Ga$ is a state. 
By linearity,  
\[
\mu\star\omega= \omega(1) \mu
\]
holds for
$
\omega\in \operatorname{span}(\Gamma)
$. By the density of $\Ga$,  it also holds for all $\omega\in \scr M_*$.
Thus, $\mu$ is a right-invariant mean on $\mathscr{M}$. This completes the proof for part \eqref{left covariant}.

Similarly, one can prove part \eqref{right covariant}.
\end{proof}

To conlude this paper, we discuss some known results in the field as direct applications of Theorem \ref{Gamma-covariant-implies-RIM}.


\begin{cor}[Corollary 7.3 of \cite{CGZ_psd_amn}] \label{When-reduced-group-C-star-alg-is- app.amn}
    Let $G$ be a discrete group. If $C_{r}^{*}(G)$ is approximately amenable, then $G$ is amenable. 
\end{cor}
\bp
By Example \ref{KIDS-example}, the quantum group $( \ell^{\infty}(G), \Delta, \phi, \phi)$ is a  right covariant over  $\Gamma=\{\delta_g : g\in G\}$. 
By \cite[Example 8.1.4]{Kadison_Fundamentals_II}, the assignment $\tau(A)=\langle A \delta_{e}, \delta_{e} \rangle$ 
defines a faithful normal tracial state on $ VN(\mathbb{G})=VN(G)$. 
From Theorem \ref{Gamma-covariant-implies-RIM},  $\ell^{\infty}(G)$ has a LIM. Thus, $G$ is amenable.
\ep

\begin{remark} \label{When-full-group-C-star-alg-is- app.amn}
    Let $G$ be a discrete group. If the full group C*-algebra $C^{*}(G)$ is  approximately amenable, then $G$ is amenable. 
\end{remark}
\bp
 By \cite[Theorem VII.2.5]{Davidson_C*}, there is a cotinuous homomorphism from $\lambda : C^{*}(G)$ onto $C_{r}^{*}(G)$, 
 which is continuous. 
 If $C^{*}(G)$ is approximately amenable, then so is $C_{r}^{*}(G)$. 
 By  Corollary \ref{When-reduced-group-C-star-alg-is- app.amn}, $G$ is amenable.
\ep

\begin{cor} \label{Full-reduced-group-C-star-alg-of-free-group-is-NOT-app.amn}
   For any natural number $d \geq 2$, neither of the group C*-algebras $ C^{*}(\mathbb{F}_d)$ and $ C_{r}^{*}(\mathbb{F}_d)$ is approximately amenable. 
\end{cor}

We now generalize \cite[Propositions~4.1]{GZ_psd} and \cite[(iii)$\Rightarrow$(i) of Propositions~3.1]{GZ_psd} (see also \cite{Zhang-app-psd-amn}) to the quantum group setting.
\begin{thm} \label{appr-amn-lcptqg}

Let $\mathbb{G}=(\mathscr{M}, \Delta, \phi, \psi)$ be a locally compact quantum group and  $\Gamma\subseteq \mathcal{S}(\mathscr{M}_{*})^+$ be regular. Suppose that $\mathbb{G}$ is left (resp. right) covariant over $\Gamma$.
    \begin{enumerate}[(i)]
        \item If $\mathscr{M}_{*}$ is pseudo-amenable, then $\mathscr{M}$ has a RIM (resp. LIM).
        \item If $\mathscr{M}_{*}$ is approximately amenable and has a bounded approximate identity, then $\mathscr{M}$ has a RIM (resp. LIM). 
    \end{enumerate}
\end{thm}

\bp
\begin{enumerate}[(i)]
    \item Suppose  that
    $$d_{\alpha}=\sum_{n=1}^{n_\alpha} a_{n}^{\alpha} \otimes b_{n}^{\alpha} \in \mathscr{M}_{*} \otimes \mathscr{M}_{*}$$
    is an approximate diagonal for $\mathscr{M}_{*}$. By Equation \eqref{multiplicative-unit}, $1 \in \mathscr{M}$ is a multiplicative linear functional on $\mathscr{M}_{*}$. Define 
    $$f_{\alpha} = (\mathrm{id} \otimes 1) d_{\alpha}=\sum_{n=1}^{n_\alpha} a_{n}^{\alpha} \langle 1, b_{n}^{\alpha} \rangle \in \mathscr{M}_{*}\,.$$

    Let $\omega \in \Gamma$. We have
$$(\mathrm{id} \otimes 1)(\omega \cdot d_{\alpha})=\sum_{n=1}^{n_\alpha}  \omega \star a_{n}^{\alpha} \langle 1, b_{n}^{\alpha} \rangle =\omega \star f_{\alpha}, $$
and
$$(\mathrm{id} \otimes 1)(d_{\alpha} \cdot \omega)=\sum_{n=1}^{n_\alpha}  a_{n}^{\alpha} \langle 1, b_{n}^{\alpha} \star \omega \rangle =\sum_{n=1}^{n_\alpha}  a_{n}^{\alpha} \langle 1, b_{n}^{\alpha}  \rangle \langle 1, \omega \rangle=f_{\alpha} \langle 1, \omega \rangle =f_{\alpha}.  $$

Hence,
$$\omega \star f_{\alpha}-f_{\alpha} = (\mathrm{id} \otimes 1)[ \omega \cdot d_{\alpha} - d_{\alpha} \cdot \omega] .$$
Denote the norm of $\scr M_*$ by $\|\cdot\|_1$. 
We derive
\begin{equation} \label{Enock-equation}
    \| \omega \star f_{\alpha}-f_{\alpha} \|_{1} =\| (\mathrm{id} \otimes 1)(\omega \cdot d_{\alpha} - d_{\alpha} \cdot \omega) \| 
    \leq  \| \omega \cdot d_{\alpha} - d_{\alpha} \cdot \omega \|_1 \longrightarrow 0
\end{equation}
Since,$(\pi (d_{\alpha}))$ is an approximate identity, we have
\begin{eqnarray*}
    \langle 1,  f_{\alpha} \rangle =  \langle 1,  \sum_{n=1}^{n_\alpha} a_{n}^{\alpha} \langle 1, b_{n}^{\alpha} \rangle \rangle = \sum_{n=1}^{n_\alpha} \langle 1, a_{n}^{\alpha}  \rangle  \langle 1, b_{n}^{\alpha} \rangle = \sum_{n=1}^{n_\alpha} \langle 1, a_{n}^{\alpha}  \star  b_{n}^{\alpha}  \rangle = \langle 1, \pi (d_{\alpha}) \rangle \longrightarrow 1\, .
\end{eqnarray*}

Thus, without loss of generality, we may assume that $\mathrm{inf} \{ \|f_{\alpha}\|_{1} \} > 0$. 
The absolute value fucntional $|f_{\alpha}|$ of $f_{\alpha}$ satisfies $\| ~|f_{\alpha}|~ \|_1=\|f_{\alpha} \|_1$. 
 Then,  $\omega_{\alpha} = \frac{|f_{\alpha}|}{\|f_{\alpha} \|_1} \in \mathscr{M}_{*}^{+}$ has norm equal to 1. By Lemma \ref{essntial-inequality-quantum-covariant}, we have
\begin{equation*} 
    \| \omega \star \omega_{\alpha}-\omega_{\alpha} \|_{1} 
    \leq \, 2\,  \| \omega \star f_{\alpha}-f_{\alpha} \|_{1}^{1/2} \longrightarrow 0
\end{equation*}
for all $\omega\in \Ga$ by Equation \eqref{Enock-equation}.

Then a standard argument (see the proof of Theorem~\ref{Gamma-covariant-implies-RIM}) concludes that a weak* cluster point $\mu \in \scr M^*$ of $(\omega_\alpha)$ is a RIM of $\scr M$.


\item Applying \cite[Proposition 10]{Zhang-app-psd-amn}, we obtain this part as a consequence of part (i).
\end{enumerate}
\ep

\bibliographystyle{unsrtnat}

\Addresses

\end{document}